\documentclass[article,10pt]{elsarticle}

\usepackage{amssymb}
\usepackage{amsmath}
\usepackage{amsfonts}
\usepackage{amsthm}
\usepackage{graphics,graphicx}
\usepackage{float}
\usepackage{mathtools}
\usepackage[dvipsnames]{xcolor}
\usepackage{tikz}
\usepackage[english]{babel}

\journal{Arxiv}
\newtheorem{theorem}{Theorem}
\newtheorem{lemma}{Lemma}

\newtheorem{corollary}{Corollary}
\newtheorem{example}{Example}
\newtheorem{conjecture}{Conjecture}

\newtheorem{problem}{Problem}
\begin{document}

\begin{frontmatter}

\title{On strong nodal domains for eigenfunctions of Hamming graphs}

\author[1]{Alexandr Valyuzhenich\corref{cor1}}
\cortext[cor1]{Corresponding author}
\ead{graphkiper@mail.ru}

\author[2]{Konstantin Vorob'ev}
\ead{konstantin.vorobev@gmail.com}

\address[1]{Postdoctoral Research Station of Mathematics, Hebei Normal University, Shijiazhuang, China}

\address[2]{Institute of Mathematics and Informatics, Bulgarian Academy of Sciences, 8 Acad G. Bonchev str., Sofia, 1113, Bulgaria}

\begin{abstract}
The Laplacian matrix of the $n$-dimensional hypercube has $n+1$ distinct eigenvalues $2i$, where $0\leq i\leq n$.
In 2004, B\i y\i ko\u{g}lu, Hordijk, Leydold, Pisanski and Stadler initiated the study of eigenfunctions of hypercubes with the minimum number of weak and strong nodal domains.
In particular, they proved that for every $1\leq i\leq \frac{n}{2}$ there is an eigenfunction of the hypercube
with eigenvalue $2i$ that have exactly two strong nodal domains.
Based on computational experiments, they conjectured that the result also holds for all $1\leq i\leq n-2$.
In this work, we confirm their conjecture for $i\leq \frac{2}{3}(n-\frac{1}{2})$ if $i$ is odd
and for $i\leq \frac{2}{3}(n-1)$ if $i$ is even.
We also consider this problem for the Hamming graph $H(n,q)$, $q\geq 3$ (for $q=2$, this graph coincides with the $n$-dimensional hypercube), 
and obtain even stronger results for all $q\geq 3$.
\end{abstract}

\begin{keyword}
hypercube \sep Hamming graph \sep strong nodal domain \sep eigenfunction
\vspace{\baselineskip}
\MSC[2010] 05C50 \sep 05B30
\end{keyword}

\end{frontmatter}

\section{Introduction}\label{Sec:Intro}
We consider finite undirected graphs without loops and multiple edges.
Let $G$ be a graph with vertex set $V$ and let $f$ be a real-valued function on $V$.
A {\em positive (negative) strong nodal domain} of $f$ is a maximal connected induced subgraph of $G$ 
on vertices $x\in V$ with $f(x)>0$ ($f(x)<0$).
A {\em positive (negative) weak nodal domain} of $f$ is a maximal connected induced subgraph of $G$ 
on vertices $x\in V$ with $f(x)\geq 0$ ($f(x)\leq 0$)  that contains at least one non-zero vertex.
The number of strong and weak nodal domains of $f$ is denoted by $\mathrm{SND}(f)$ and $\mathrm{WND}(f)$ respectively.

Let $G$ be a graph with $n$ vertices.
A symmetric real $n\times n$ matrix $M$ is called a {\em generalized Laplacian} of $G$ if
it has non-positive off-diagonal entries and two distinct vertices $x$ and $y$ are adjacent if and only if $M_{x,y}<0$.
There are no restrictions on the diagonal entries of $M$.

In 2001, Davies, Gladwell, Leydold and Stadler \cite{DGLS01} established upper bounds on the number of weak and strong nodal domains for eigenfunctions of generalized Laplacians.

\begin{theorem}[\cite{DGLS01}, Theorems 1 and 2]\label{Th:NDT}
Let $M$ be a generalized Laplacian of a connected graph with $n$ vertices.
Then any eigenfunction $f$ corresponding to the $k$-th eigenvalue of $M$
with multiplicity $r$ has at most $k$ weak nodal domains and $k+r-1$ strong nodal domains:
\begin{equation}\label{}
\mathrm{WND}(f)\leq k\qquad\text{and}\qquad\mathrm{SND}(f)\leq k+r-1.
\end{equation}
\end{theorem}
Theorem \ref{Th:NDT} is a discrete analogue of Courant's nodal domain theorem for elliptic
operators on manifolds.
Various versions of Theorem \ref{Th:NDT} with different proofs can be found in \cite{Colin93,DR99,F93,P88,H96}, 
and detailed discussions of these versions can be found in \cite{BLS07,DGLS01}.
The results of Gantmacher and Krein on oscillation matrices \cite{GK} imply that the eigenvalues of a generalized Laplacian of a path are all simple, and that
any eigenfunction corresponding to the $k$-th eigenvalue has exactly $k$ weak nodal domains and $k$ strong nodal domains.
Therefore, both bounds from Theorem \ref{Th:NDT} cannot be improved without additional assumptions.
On the other hand, these bounds are not sharp for all graphs. 
In particular, they have been improved for trees \cite{B03}, cographs \cite{BLS07} and hypercubes \cite{BHLPS04}.
Non-trivial lower bounds on the number of strong nodal domains for eigenfunctions of generalized Laplacians 
can be found in \cite{B08,XY12}.
Recently, Ge and Liu \cite{GL23} extended Theorem \ref{Th:NDT} and the lower bounds from \cite{B08,XY12} to arbitrary symmetric matrices.
For more information on Theorem \ref{Th:NDT} and related questions we refer the reader to \cite{BLS07}.

The {\em Hamming graph} $H(n,q)$ is defined as follows.
The vertex set of $H(n,q)$ is $\mathbb{Z}_{q}^n$, and two vertices are adjacent if they differ in exactly one coordinate.
The Hamming graph $H(n,2)$ is also known as the $n$-dimensional hypercube or $n$-cube.
The Laplacian matrix of $H(n,q)$ has $n+1$ distinct eigenvalues $\lambda_i(n,q)=q\cdot i$ 
with corresponding multiplicities $\binom{n}{i}\cdot (q-1)^i$, where $0\leq i\leq n$.
An eigenfunction of $H(n,q)$ is an eigenfunction of its Laplacian.

In 2004, B\i y\i ko\u{g}lu, Hordijk, Leydold, Pisanski and Stadler \cite{BHLPS04} initiated the study of the following two problems for eigenfunctions of hypercubes.

\begin{problem}\label{Problem:1}
Let $1\leq i\leq n-1$.
Find the minimum number of weak nodal domains of an eigenfunction of $H(n,2)$ corresponding to the eigenvalue $2i$.
\end{problem}

\begin{problem}\label{Problem:2}
Let $1\leq i\leq n-1$.
Find the minimum number of strong nodal domains of an eigenfunction of $H(n,2)$ corresponding to the eigenvalue $2i$.
\end{problem}

In the case of weak nodal domains they proved the following result.

\begin{theorem}[\cite{BHLPS04}, Theorem 3]\label{Th:Weak-H}
For every $1\leq i\leq n-1$, there is an eigenfunction $f$ of $H(n,2)$ 
with eigenvalue $2i$ such that $\mathrm{WND}(f)=2$.   
\end{theorem}

For strong nodal domains the situation is much more complicated.
In this case, B\i y\i ko\u{g}lu et al. \cite{BHLPS04} proved the following results.

\begin{theorem}[\cite{BHLPS04}, Theorem 4]\label{Th:Strong-H}
For every $1\leq i\leq \frac{n}{2}$, there is an eigenfunction $f$ of $H(n,2)$
with eigenvalue $2i$ such that $\mathrm{SND}(f)=2$.   
\end{theorem}

\begin{theorem}[\cite{BHLPS04}, Theorem 5]\label{Th:Strong-N-1}
For any eigenfunction $f$ of $H(n,2)$, $n\geq 3$, with eigenvalue $2(n-1)$ we have $\mathrm{SND}(f)\geq n$.   
\end{theorem}

Based on computational experiments, B\i y\i ko\u{g}lu et al. \cite{BHLPS04} conjectured that Theorem \ref{Th:Strong-H} can be extended as follows.

\begin{conjecture}[\cite{BHLPS04}, Conjecture 1]\label{Conj:Strong-H}
For every $1\leq i\leq n-2$, there is an eigenfunction $f$ of $H(n,2)$ 
with eigenvalue $2i$ such that $\mathrm{SND}(f)=2$.  
\end{conjecture}

In this work, we confirm Conjecture \ref{Conj:Strong-H} for $i\leq \frac{2}{3}(n-\frac{1}{2})$ if $i$ is odd
and for $i\leq \frac{2}{3}(n-1)$ if $i$ is even.
For $q\geq 3$, we consider the following natural generalization of Problem \ref{Problem:2}.

\begin{problem}\label{Problem:3}
Let $q\geq 3$ and $1\leq i\leq n$.
Find the minimum number of strong nodal domains of an eigenfunction of $H(n,q)$ corresponding to the eigenvalue $q\cdot i$.
\end{problem}

In this work, we solve Problem \ref{Problem:3} for all $q\geq 3$ except for the case $q=3$ and $i=n$.
Namely, we prove that for every $q\geq 3$ and every $1\leq i\leq n-1$ there is an eigenfunction $f$ of $H(n,q)$ 
with eigenvalue $q\cdot i$ such that $\mathrm{SND}(f)=2$.
We also show that for every $q\geq 4$ there is an eigenfunction $f$ of $H(n,q)$ 
with eigenvalue $q\cdot n$ such that $\mathrm{SND}(f)=2$.

The paper is organized as follows. In Section \ref{Sec:Def}, we introduce basic definitions.
In Section \ref{Sec:Prelim}, we give preliminary results.
In Section \ref{Sec:EP}, we provide background on equitable 2-partitions of Hamming graphs.
In Section \ref{Sec:Small-Eigenfunctions}, we define two special eigenfunctions of $H(2,2)$ and three special eigenfunctions of $H(3,3)$.
In Section \ref{Sec:q=2}, we confirm Conjecture \ref{Conj:Strong-H} for $i\leq \frac{2}{3}(n-\frac{1}{2})$ if $i$ is odd
and for $i\leq \frac{2}{3}(n-1)$ if $i$ is even.
In Section \ref{Sec:q=3}, we prove that for every $n\geq 2$ and every $1\leq i\leq n-1$ there is an eigenfunction $f$ of $H(n,3)$ 
with eigenvalue $3i$ such that $\mathrm{SND}(f)=2$.
In Section \ref{Sec:q>=4}, we prove that for every $q\geq 4$ and every $1\leq i\leq n$ there is an eigenfunction $f$ of $H(n,q)$ 
with eigenvalue $q\cdot i$ such that $\mathrm{SND}(f)=2$.
In Section \ref{Sec:Conclusion}, we discuss open problems and future directions.

\section{Basic definitions}\label{Sec:Def}

For a function $f:X\rightarrow \mathbb{R}$, denote
$S_+(f)=\{x\in X:f(x)>0\}$ and $S_-(f)=\{x\in X:f(x)<0\}$.

The {\em Laplacian matrix} of a graph $G$, denoted by $L$, is defined by $L=D-A$, 
where $D$ is the diagonal matrix formed from the vertex degrees of $G$
and $A$ is the adjacency matrix of $G$.

Let $G$ be a graph with vertex set $V$ and let $\lambda$ be an eigenvalue of its Laplacian matrix.
A function $f:V\rightarrow \mathbb{R}$ is called a {\em $\lambda$-eigenfunction} of $G$ if $f\not\equiv 0$ and the equality
\begin{equation}\label{Eq:Eigenfunction}
\lambda\cdot f(x)=\sum_{y\in{N(x)}}(f(x)-f(y))
\end{equation}
holds for any vertex $x\in V$, where $N(x)$ is the set of all neighbors of $x$ in $G$.

Let $G_1=(V_1,E_1)$ and $G_2=(V_2,E_2)$ be two graphs.
The {\em Cartesian product} $G_1\square G_2$ of graphs $G_1$ and $G_2$ is defined as follows.
The vertex set of $G_1\square G_2$ is $V_1\times V_2$;
and any two vertices $(x_1,y_1)$ and $(x_2,y_2)$ are adjacent if and only if either
$x_1=x_2$ and $y_1$ is adjacent to $y_2$ in $G_2$, or
$y_1=y_2$ and $x_1$ is adjacent to $x_2$ in $G_1$.

Suppose $G_1=(V_1,E_1)$ and $G_2=(V_2,E_2)$ are two graphs. Let $f_1:V_1\rightarrow{\mathbb{R}}$ and $f_2:V_2\rightarrow{\mathbb{R}}$.
Denote $G=G_1\square G_2$.
We define the {\em tensor product} $f_1\otimes f_2$  on the vertices of $G$ by the following rule:
$$(f_1\otimes f_2)(x,y)=f_1(x)f_2(y)$$ for $(x,y)\in V(G)=V_1\times V_2$.

For a positive integer $n$, denote $[n]=\{1,\ldots,n\}$.
For $t\in [n]$, we denote by $e_t$ the tuple from $\mathbb{Z}_q^n$ whose $i$-th entry is equal to $0$ if $i\neq t$ and equal to $1$ if $i=t$.
For a set $X\subseteq \mathbb{Z}_q^n$ and $t\in [n]$, denote $X+e_t=\{x+e_t:x\in X\}$, where $+$ means addition in $\mathbb{Z}_q^n$.

Let $k\in [n]$ and let $a_1,\ldots,a_k\in \mathbb{Z}_q$.
For a function $f$ on $\mathbb{Z}_q^n$, we denote by $f_{a_1,\ldots,a_k}$ its restriction to the set
$\{x\in \mathbb{Z}_q^n:x_{n-k+1}=a_1,x_{n-k+2}=a_2,\ldots,x_n=a_k\}$.

In what follows, we will often consider sums of the form $x_{i_1}+x_{i_2}+\cdots+ x_{i_k}$, 
where $x_{i_1},x_{i_2},\ldots,x_{i_k}$ are entries of a vertex of $H(n,q)$. In all cases $+$ means addition in $\mathbb{Z}_q$.

\section{Preliminaries}\label{Sec:Prelim}
The following fact about the Cartesian product of graphs is well-known (for example, see \cite[Rule 7.9]{N18}).
\begin{lemma}\label{L:Product}
Let $G_1$ and $G_2$ be two graphs. If $f_1$ is a $\lambda$-eigenfunction of $G_1$ and $f_2$ is a $\mu$-eigenfunction of $G_2$, then
$f_1\otimes f_2$ is a ($\lambda+\mu$)-eigenfunction of $G_1\square G_2$.
\end{lemma}

Using Lemma \ref{L:Product} for $G_1=H(m,q)$ and $G_2=H(n,q)$, we immediately obtain the following result.

\begin{corollary}\label{Corollary:Product}
Let $f_1$ be a $\lambda_i(m,q)$-eigenfunction of $H(m,q)$ and let $f_2$ be a $\lambda_j(n,q)$-eigenfunction of $H(n,q)$. 
Then $f_1\otimes f_2$ is a $\lambda_{i+j}(m+n,q)$-eigenfunction of $H(m+n,q)$.
\end{corollary}

We will also need the following result.

\begin{lemma}\label{L:N-E}
Let $g$ be a real-valued function on $\mathbb{Z}_q^n$ and let $h$ be a function on $\mathbb{Z}_q$ such that $h(x)=1$ for all $x\in \mathbb{Z}_q$.
If $\mathrm{SND}(g)=2$, then $\mathrm{SND}(g\otimes h)=2$.
\end{lemma}
\begin{proof}
Let $f=g\otimes h$.
By the definition of the tensor product of functions, we have

\begin{equation}\label{Eq:N-E-1}
S_{+}(f)=S_{+}(g)\times \mathbb{Z}_q  
\end{equation}

and

\begin{equation}\label{Eq:N-E-2}
S_{-}(f)=S_{-}(g)\times \mathbb{Z}_q  
\end{equation}

Since $\mathrm{SND}(g)=2$, the subgraphs of $H(n,q)$ induced by $S_{+}(g)$ and $S_{-}(g)$ are both connected.
Using (\ref{Eq:N-E-1}) and (\ref{Eq:N-E-2}), 
we obtain that the subgraphs of $H(n+1,q)$ induced by $S_{+}(f)$ and $S_{-}(f)$ are also connected.
Therefore, we have $\mathrm{SND}(f)=2$.
\end{proof}

\begin{corollary}\label{Cor:N-E-0}
If there exists a $\lambda_{i}(n,q)$-eigenfunction $g$ of $H(n,q)$ such that $\mathrm{SND}(g)=2$,
then there exists a $\lambda_{i}(n+1,q)$-eigenfunction $f$ of $H(n+1,q)$ such that $\mathrm{SND}(f)=2$.
\end{corollary}
\begin{proof}
Let $h$ be a function on $\mathbb{Z}_q$ such that $h(x)=1$ for all $x\in \mathbb{Z}_q$
and let $f=g\otimes h$.
Lemma \ref{L:N-E} implies that $\mathrm{SND}(f)=2$.
Note that $h$ is a $\lambda_{0}(1,q)$-eigenfunction of $H(1,q)$.
Therefore, $f$ is a $\lambda_{i}(n+1,q)$-eigenfunction of $H(n+1,q)$ due to Corollary \ref{Corollary:Product}.
\end{proof}

Using Corollary \ref{Cor:N-E-0}, we immediately obtain the following result.

\begin{corollary}\label{Cor:N-E}
Suppose that there exists a $\lambda_{i}(n_0,q)$-eigenfunction $g$ of $H(n_0,q)$ such that $\mathrm{SND}(g)=2$.
Then for every $n>n_0$ there exists a $\lambda_{i}(n,q)$-eigenfunction $f$ of $H(n,q)$ such that $\mathrm{SND}(f)=2$.
\end{corollary}

In what follows, we will need the following fact.

\begin{lemma}\label{L:U_1}
For every $n\geq 1$ and every $q\geq 2$, 
there exists a $\lambda_{1}(n,q)$-eigenfunction $f$ of $H(n,q)$ such that $\mathrm{SND}(f)=2$.
\end{lemma}
\begin{proof}
Firstly, let us prove the statement for $n=1$.
We define a function $f$ on $\mathbb{Z}_q$ by the following rule:
$$
f(x)=\begin{cases}
1,&\text{if $x=0$;}\\
-1,&\text{if $x=1$;}\\
0,&\text{otherwise.}
\end{cases}
$$

It is easy to see that $f$ is a $\lambda_{1}(1,q)$-eigenfunction of $H(1,q)$ and $\mathrm{SND}(f)=2$.
Therefore, we have proved the lemma for $n=1$.
The proof for $n>1$ follows from Corollary \ref{Cor:N-E}.
\end{proof}

\section{Equitable partitions}\label{Sec:EP}
In this section, we provide background on equitable 2-partitions of Hamming graphs.
In particular, we discuss in detail two special constructions of equitable 2-partitions of Hamming graphs.
We will use these constructions in Sections \ref{Sec:q=2} and \ref{Sec:q=3}.

\subsection{Equitable partitions: definitions and examples}
Let $G$ be a graph with vertex set $V$.
An $r$-partition $(C_1,C_2,\ldots,C_r)$ of $V$ is called {\em equitable} if
for any $i,j\in [r]$ there is a constant $s_{i,j}$ such that any vertex of $C_i$ has exactly $s_{i,j}$ neighbors in $C_j$.
The matrix $S=(s_{i,j})_{1\leq i,j\leq r}$ is the {\em quotient matrix} of the partition.
Double counting the edges between $C_i$ and $C_j$ in $G$ yields the following equality:

\begin{equation}\label{Eq:1}
|C_i|\cdot s_{i,j}=|C_j|\cdot s_{j,i}
\end{equation}

\begin{example}\label{Ex:Eq}
Let $D_2=\{(1,0,0), (0,1,0), (0,0,1), (1,1,0), (1,0,1), (0,1,1)\}$ and $D_1=\{(0,0,0), (1,1,1)\}$. Then $(D_1,D_2)$ is an equitable $2$-partition of $H(3,2)$ 
with the quotient matrix 
$\begin{pmatrix}
0 & 3\\
1 & 2\\
\end{pmatrix}$
(see Figure \ref{fig:M1}).
\end{example}

	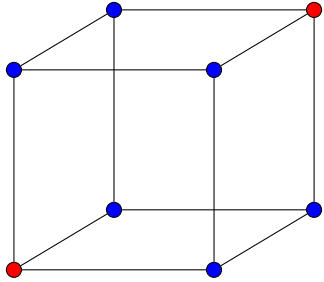
\begin{figure}
	\begin{center}
	\begin{minipage}{.45\linewidth}
		
		\begin{tikzpicture}[scale=2.66,
			mycirc/.style={circle, minimum size=0.2cm}]
			\tikzstyle{every node}=[draw,circle,fill=white,minimum size=0pt,
			inner sep=0pt]

			
			\node[mycirc,shape=circle,draw=black,text width=0.2cm, fill=red, ] (000) at (0,0) {};
			
			\node[mycirc,shape=circle,draw=black,text width=0.2cm, fill=blue, ] (100) at (1,0) {};
			\node[mycirc,shape=circle,draw=black,text width=0.2cm, fill=blue, ] (010) at (0,1) {};
			\node[mycirc,shape=circle,draw=black,text width=0.2cm, fill=blue, ] (110) at (1,1) {};

			\node[mycirc,shape=circle,draw=black,text width=0.2cm, fill=blue, ] (001) at (0+0.5,0+0.3) {};
			
			\node[mycirc,shape=circle,draw=black,text width=0.2cm, fill=blue, ] (101) at (1+0.5,0+0.3) {};
			\node[mycirc,shape=circle,draw=black,text width=0.2cm, fill=blue, ] (011) at (0+0.5,1+0.3) {};
			\node[mycirc,shape=circle,draw=black,text width=0.2cm, fill=red, ] (111) at (1+0.5,1+0.3) {};

			
			\path[draw=black]  (000) edge (100);
			\path[draw=black]  (000) edge (010);
			\path[draw=black]  (110) edge (010);
			\path[draw=black]  (110) edge (100);
			
			\path[draw=black]  (001) edge (101);
			\path[draw=black]  (001) edge (011);
			\path[draw=black]  (111) edge (011);
			\path[draw=black]  (111) edge (101);
			
			\path[draw=black]  (000) edge (001);
			\path[draw=black]  (100) edge (101);
			\path[draw=black]  (010) edge (011);
			\path[draw=black]  (110) edge (111);

		\end{tikzpicture}

	\end{minipage}\hfill

\end{center}
\caption{Partition $(D_1,D_2)$ in $H(3,2)$} \label{fig:M1}
\end{figure}

A general information on equitable partitions can be found in \cite[Chapter~5]{Godsil}.
For more information on equitable 2-partitions of Hamming graphs we refer the reader to a survey \cite{BKMTV21}.

\subsection{Multiplication construction}

For our arguments we will need the following multiplication construction of equitable 2-partitions in a Hamming graph $H(n,2)$, that was firstly proposed by Fon-Der-Flaass in \cite[Proposition 1]{FDF07SMJ}. For a generalization of this construction for arbitrary number of cells and arbitrary $q$ see also \cite[Proposition 4.4]{BKMTV21}.

\begin{lemma}[\cite{FDF07SMJ}, Proposition 1]\label{L:Multiple-Constr}
Let $(D_1,D_2)$ be an equitable $2$-partition of $H(n,2)$ with the quotient matrix $S$. For integer $k\geq 1$, consider 
a partition $(C_1,C_2)$ of $H(kn,2)$ defined in the following way:
\vspace{1.5mm}

for any $x=(x_1,x_2,\dots, x_{kn}) \in \mathbb{Z}_2^{kn}$ and $i \in \{1,2\}$, we have $x \in C_i$ if and only if 

$$\bigl(\sum_{i=1}^{k}{x_i},\sum_{i=1}^{k}{x_{k+i}}, \dots, \sum_{i=1}^{k}{x_{(n-1)k+i}} \bigr) \in D_i .$$

Then $(C_1, C_2)$ is an equitable $2$-partition of $H(kn,2)$ with the quotient matrix $kS$.

\end{lemma}

\subsection{Construction of type A}

We start from the following equitable $2$-partition $(D_1, D_2)$ of $H(3,2)$ from Example \ref{Ex:Eq}:
$$D_1=\{(0,0,0), (1,1,1)\},\ D_2=\{(1,0,0), (0,1,0), (0,0,1), (1,1,0), (1,0,1), (0,1,1)\}.$$
As we noted in Example \ref{Ex:Eq}, the quotient matrix of the partition is $\left(\begin{array}{cl}
	0 & 3\\

	1 & 2
\end{array}\right)$.
Let us also notice, that the subgraph of $H(3,2)$ induced by the set $D_2$ is connected, 
and by the set $D_1$ is not connected.
For any integer $k\geq 1$, we apply the multiplication construction to our partition and obtain a new equitable 2-partition $(C_1,C_2)$ of $H(3k,2)$ with the quotient matrix $\left(\begin{array}{cl}
	0 & 3k\\
	
	k & 2k
\end{array}\right)$, that will be referred as the partition of type $A$ in our further arguments (for $k=1$ we have $C_1=D_1$ and $C_2=D_2$).

In what follows, we will need the following two results.

\begin{lemma}\label{L:Translation}
Let $k\geq 1$ and let $(C_1,C_2)$ be an equitable $2$-partition of $H(3k,2)$ of type A.
Let $C_3=C_1+e_t$ and $C_4=C_2+e_t$, where $t\in [3k]$. Then the following statements hold:

\begin{enumerate}
  \item $C_1\subset C_4$.
  
  \item $C_3\subset C_2$.
  
  \item $C_2\cap C_4\neq \emptyset$.
\end{enumerate}

\end{lemma}
\begin{proof}
1. Suppose that there exists a vertex $x\in C_1$ such that $x\not\in C_4$. Then $x\in C_3$. Therefore, $x=y+e_t$ for some vertex $y\in C_1$.
This implies that $x$ and $y$ are adjacent in $H(3k,2)$.
This contradicts with $s_{1,1}=0$.
Thus, we have $C_1\subseteq C_4$.
On the other hand, the equality (\ref{Eq:1}) for $i=1$ and $j=2$ implies that $|C_2|=3|C_1|$. Hence, $|C_4|=3|C_1|$. Therefore, $C_1\subset C_4$.

2. The proof for the second case is similar.

3. Since $C_1\subset C_4$, any vertex from $C_4\setminus C_1$ belongs to $C_2$. Therefore, the set $C_2\cap C_4$ is not empty.
\end{proof}

\begin{lemma}\label{L:Connectivity}
Let $k\geq 1$ and let $(C_1,C_2)$ be an equitable $2$-partition of $H(3k,2)$ of type A.
Then the subgraph of $H(3k,2)$ induced by $C_2$ is connected.
\end{lemma}
\begin{proof}
	Since $(C_1,C_2)$ is a partition of type $A$, the set $C_2$ can be described as follows:
	$$x=(x_1, x_2,\dots, x_k,x_{k+1}, x_{k+2},\dots, x_{2k}, x_{2k+1}, x_{2k+2},\dots, x_{3k} ) \in C_2 $$
	if and only if
	\begin{equation}\label{Eq:multi-constr}
 \bigl(\sum_{i=1}^{k}{x_i},\sum_{i=k+1}^{2k}{x_{i}},  \sum_{i=2k+1}^{3k}{x_i} \bigr) \in D_2,  
	\end{equation}
where $D_2=\bigl\{(1,0,0), (0,1,0), (0,0,1), (1,1,0), (1,0,1), (0,1,1)\bigr\}. $
	As it was defined in the beginning, the summation in (\ref{Eq:multi-constr}) is modulo $2$.
	
	In case $k=1$, we have that $C_2=D_2$. Consequently, the subgraph of $H(3,2)$ induced by $C_2$ is connected.
	In the rest of the proof consider the case $k\geq 2$. Clearly, $C_2$ is a union of six sets $M_{100}$, $M_{010}$, $M_{001}$, $M_{110}$,  $M_{101}$ and $M_{011}$ corresponding to different elements of $D_2$.
	
	Let us denote by $G$ the subgraph of $H(3k,2)$ induced by $C_2$. Our next step is to prove that for any two distinct vertices from $M_{100}$ there exist a path in $G$, that connects them. In order to do it, we will show how to build a path from $y$ to $z$ in $G$,
	where $y$ and $z$ are two arbitrary elements of $M_{100}$, i.e.   
	$$y=(y_1, y_2,\dots, y_k,y_{k+1}, y_{k+2},\dots, y_{2k}, y_{2k+1}, y_{2k+2},\dots, y_{3k} ), $$
	$$z=(z_1, z_2,\dots, z_k,y_{k+1}, z_{k+2},\dots, z_{2k}, z_{2k+1}, z_{2k+2},\dots, z_{3k} ), $$
	$$ \sum_{i=1}^{k}{y_i}=\sum_{i=1}^{k}{z_i}=1,\,\sum_{i=k+1}^{2k}{y_i}=\sum_{i=k+1}^{2k}{z_i}=0,\,\sum_{i=2k+1}^{3k}{y_i}=\sum_{i=2k+1}^{3k}{z_i}=0.$$
	By the definition of a hypercube, one edge in the graph corresponds to changing a value at some coordinate position. Therefore, any path may be encoded by a sequence of coordinate positions in which we change a value. 
	
	For $t=0,1,2$, let us define $I_t=\{i \in \{tk+1,tk+2,\dots (t+1)k\}| y_i\neq z_i \}$. Now let us find the path. At first, we change step by step entries of $y$ in coordinate positions from $2k+1$ to $3k$, in which $y$ differs from $z$. In other words, we do it in coordinate positions $I_2$. Clearly, at each step we are in $M_{100}$ or $M_{101}$. As a result, we are at vertex 
		$$y'=(y_1, y_2,\dots, y_k,y_{k+1}, y_{k+2},\dots, y_{2k}, z_{2k+1}, z_{2k+2},\dots, z_{3k} )\in M_{100}. $$
	Then we continue the procedure for a set of positions $I_1$. Again, at each step we are in $C_2$ (more precisely, in $M_{100}$ or $M_{110}$). As a result, we find a path from $y$ to 
	$$y''=(y_1, y_2,\dots, y_k,z_{k+1}, z_{k+2},\dots, z_{2k}, z_{2k+1}, z_{2k+2},\dots, z_{3k} )\in M_{100}.$$ 
	Evidently, one can not change any value in the first third of coordinate positions of $y''$ (then we will move to $M_{000}\subset C_1$). Hence, we change $(k+1)$-th entry of $y''$ to $z_{k+1}+1$ (and move to $M_{110}$) and then continue the procedure for the set $I_0$. Finally, we build a path from $y$ to 
    $$y'''=(z_1, z_2,\dots, z_k,z_{k+1}+1, z_{k+2},\dots, z_{2k}, z_{2k+1}, z_{2k+2},\dots, z_{3k} )\in M_{110}.$$ 
    The last step $z_{k+1}+1\rightarrow z_{k+1}$ finishes the path and we prove that $M_{100}$ is a subset of some connected component in $G$.
    The proof for sets $M_{010}$ $M_{001}$ is similar (one need to make an appropriate permutation of coordinate positions, and all arguments work). For sets $M_{110}$, $M_{101}$ and $M_{011}$, the scheme is also the similar. The only difference is that we first change entries from the thirds corresponding to a partial sum equal to $1$ and finish with the third corresponding to a partial sum equal to $0$.

    As it was discussed above, the subgraph of $H(3,2)$ induced by $D_2$ is connected. Consequently, there exists a path in $C_2$, that contains vertices from all the sets $M_{100}$, $M_{010}$, $M_{001}$, $M_{110}$,  $M_{101}$ and $M_{011}$. Therefore, all these sets are in one connected component in $G$, which finishes the proof.
\end{proof}

\subsection{Construction of type B}
For $n\geq 1$, $q\geq 2$ and $a\in \mathbb{Z}_q$, denote $\Gamma_a^{n,q}=\{x\in \mathbb{Z}_q^n:x_1+\ldots+x_n=a\}$.
The following fact follows directly from the definition of $\Gamma_a^{n,q}$.

\begin{lemma}\label{L:Equitable-Hyperplane}
Let $n\geq 1$ and $q\geq 2$.
Let $C_1=\Gamma_0^{n,q}$ and $C_2=\Gamma_1^{n,q}\cup\ldots\cup \Gamma_{q-1}^{n,q}$. Then $(C_1,C_2)$ is an equitable $2$-partition of $H(n,q)$ 
with the quotient matrix 
$\begin{pmatrix}
0 & n(q-1)\\
n & n(q-2)\\
\end{pmatrix}$.
\end{lemma}

We will say that a 2-partition $(C_1,C_2)$ of the vertex set of $H(n,q)$ has type B if $C_1=\Gamma_0^{n,q}$ and $C_2=\Gamma_1^{n,q}\cup\ldots\cup \Gamma_{q-1}^{n,q}$.

In what follows, we will need the following three results.
\begin{lemma}\label{L:Translation-3}
Let $(C_1,C_2)$ be an equitable $2$-partition of $H(n,3)$ of type B.
Let $C_3=C_1+e_t$ and $C_4=C_2+e_t$, where $t\in [n]$. Then the following statements hold:

\begin{enumerate}
  \item $C_1\subset C_4$.
  
  \item $C_3\subset C_2$.
  
  \item $C_2\cap C_4\neq \emptyset$.
\end{enumerate}

\end{lemma}
\begin{proof}
1. Suppose that there exists a vertex $x\in C_1$ such that $x\not\in C_4$. Then $x\in C_3$. Therefore, $x=y+e_t$ for some vertex $y\in C_1$,
i.e. $x$ and $y$ are adjacent in $H(n,3)$.
This contradicts Lemma \ref{L:Equitable-Hyperplane}.
Thus, we have $C_1\subseteq C_4$.
On the other hand, Lemma \ref{L:Equitable-Hyperplane} and the equality (\ref{Eq:1}) for $i=1$ and $j=2$ imply that $|C_2|=2|C_1|$. Hence $|C_4|=2|C_1|$. Therefore, $C_1\subset C_4$.

2. The proof for the second case is similar.

3. Since $C_1\subset C_4$, any vertex from $C_4\setminus C_1$ belongs to $C_2$. Therefore, the set $C_2\cap C_4$ is not empty.
\end{proof}

\begin{lemma}\label{L:Hyperplane}
For every $n\geq 1$, every $q\geq 2$ and every $a\in \mathbb{Z}_q$, the subgraph of $H(n,q)$ induced by $\Gamma_a^{n,q}\cup \Gamma_{a+1}^{n,q}$ is connected.
\end{lemma}
\begin{proof}
Denote $\Gamma=\Gamma_a^{n,q}\cup \Gamma_{a+1}^{n,q}$.
For $c\in \mathbb{Z}_q$, denote $T_c=\{x\in \mathbb{Z}_q^n:x_n=c\}$.
Let us prove the statement by induction on $n$.
If $n=1$, then $\Gamma$ consists of two adjacent vertices, that is, the subgraph of $H(n,q)$ induced by $\Gamma$ is connected.

Let us prove the induction step for $n\geq 2$. 
By the induction assumption, the subgraph of $H(n,q)$ induced by $\Gamma\cap T_c$ is connected for all $c\in \mathbb{Z}_q$.
Let us show that for any $c\in \mathbb{Z}_q$ there is an edge between $\Gamma\cap T_c$ and $\Gamma\cap T_{c+1}$ in $H(n,q)$.
Let us consider the vertices $x=(a-c,0,\ldots,0,c)$ and $y=(a-c,0,\ldots,0,c+1)$.
Note that $x\in \Gamma_a^{n,q}\cap T_c$ and $y\in \Gamma_{a+1}^{n,q}\cap T_{c+1}$.
In addition, $x$ and $y$ are adjacent in $H(n,q)$.
Therefore, there is an edge between $\Gamma\cap T_c$ and $\Gamma\cap T_{c+1}$ in $H(n,q)$.
This implies that the subgraph of $H(n,q)$ induced by $\Gamma$ is connected.
\end{proof}

\begin{lemma}\label{L:Connectivity-3}
Let $(C_1,C_2)$ be an equitable $2$-partition of $H(n,3)$ of type B.
Then the subgraph of $H(n,3)$ induced by $C_2$ is connected.
\end{lemma}
\begin{proof}
Since $(C_1,C_2)$ is an equitable $2$-partition of $H(n,3)$ of type B, we have $C_2=\Gamma_1^{n,3}\cup \Gamma_2^{n,3}$.
Therefore, by Lemma \ref{L:Hyperplane}, the subgraph of $H(n,3)$ induced by $C_2$ is connected.
\end{proof}

\subsection{Associated functions}

Let $G$ be a graph.
For an equitable $2$-partition $(C_1,C_2)$ of $G$ with quotient matrix 
$\begin{pmatrix}
s_{1,1} & s_{1,2}\\
s_{2,1} & s_{2,2}\\
\end{pmatrix}$,
we define its associated function $f$ on the vertex set of $G$ by the following rule: 
$$
f(x)=\begin{cases}
s_{1,2},&\text{if $x\in C_1$;}\\
-s_{2,1},&\text{if $x\in C_2$.}
\end{cases}
$$

The following fact about associated functions of equitable $2$-partitions of graphs is well-known.
We include the proof for completeness.
\begin{lemma}\label{L:Associated-Function}
Let $G$ be a graph. If $(C_1,C_2)$ is an equitable $2$-partition of $G$ with quotient matrix 
$\begin{pmatrix}
s_{1,1} & s_{1,2}\\
s_{2,1} & s_{2,2}\\
\end{pmatrix}$ and associated function $f$, 
then $f$ is an $(s_{1,2}+s_{2,1})$-eigenfunction of $G$.
\end{lemma}
\begin{proof}
Let us consider a vertex $x\in C_1$.
For any vertex $y\in C_1$ we have $f(x)-f(y)=0$, 
and for any vertex $y\in C_2$ we have $f(x)-f(y)=s_{1,2}+s_{2,1}$.
In addition, the vertex $x$ has exactly $s_{1,2}$ neighbors in $C_2$.
Therefore, the equality~(\ref{Eq:Eigenfunction}) holds for any vertex $x\in C_1$.
The proof for the case $x\in C_2$ is similar.
\end{proof}

\begin{lemma}\label{L:Associated-Corollary}
The following statements hold:
\begin{enumerate}
  
  \item Let $k\geq 1$ and let $(C_1,C_2)$ be an equitable $2$-partition of $H(3k,2)$ of type A.
  If $f$ is the associated function of $(C_1,C_2)$, then $f$ is a $\lambda_{2k}(3k,2)$-eigenfunction of $H(3k,2)$.

  \item Let $(C_1,C_2)$ be an equitable $2$-partition of $H(n,q)$ of type B.
   If $f$ is the associated function of $(C_1,C_2)$, then $f$ is a $\lambda_{n}(n,q)$-eigenfunction of $H(n,q)$.

\end{enumerate}
\begin{proof}
1. It follows from Lemma \ref{L:Associated-Function} and Lemma \ref{L:Multiple-Constr}.

2. It follows from Lemma \ref{L:Associated-Function} and Lemma \ref{L:Equitable-Hyperplane}.
\end{proof}

\end{lemma}

\section{Eigenfunctions in $H(2,2)$ and $H(3,3)$}\label{Sec:Small-Eigenfunctions}

In this section, we define five special functions: two on $H(2,2)$ and three on $H(3,3)$.
We will use these functions in Sections \ref{Sec:q=2} and \ref{Sec:q=3} to construct eigenfunctions of Hamming graphs with two strong nodal domains.

We define a function $\varphi_1$ on $\mathbb{Z}_2^2$ by the following rule:
$$
\varphi_1(x)=\begin{cases}
1,&\text{if $x=(0,0)$;}\\
-1,&\text{if $x=(1,1)$;}\\
0,&\text{otherwise.}
\end{cases}
$$

We define a function $\varphi_2$ on $\mathbb{Z}_2^2$ by the following rule:
$$
\varphi_2(x)=\begin{cases}
1,&\text{if $x=(0,1)$;}\\
-1,&\text{if $x=(1,0)$;}\\
0,&\text{otherwise.}
\end{cases}
$$

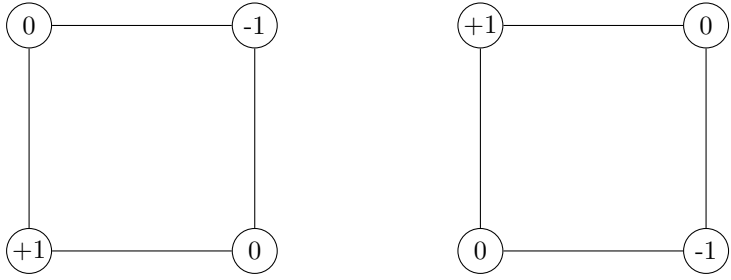
\begin{figure}
	
	\begin{center}
		\begin{minipage}{1.45\linewidth}
			
			\begin{tikzpicture}[scale=3,
				mycirc/.style={circle, minimum size=0.2cm}]
				\tikzstyle{every node}=[draw,circle,fill=white,minimum size=0pt,
				inner sep=0pt]

				
				\node[mycirc,shape=circle,draw=black,text width=0.6cm, label=center:{ +1}] (00) at (0-1,0) {};
				
				\node[mycirc,shape=circle,draw=black,text width=0.6cm,label=center:{ 0}] (10) at (1-1,0) {};
				\node[mycirc,shape=circle,draw=black,text width=0.6cm, label=center:{ 0}] (01) at (0-1,1) {};
				\node[mycirc,shape=circle,draw=black,text width=0.6cm, label=center:{ -1}] (11) at (1-1,1) {};

				\node[mycirc,shape=circle,draw=black,text width=0.6cm,label=center:{ 0} ] (007) at (2-1,0) {};
				
				\node[mycirc,shape=circle,draw=black,text width=0.6cm,label=center:{ -1}]  (107) at (3-1,0) {};
				\node[mycirc,shape=circle,draw=black,text width=0.6cm, label=center:{ +1}] (017) at (2-1,1) {};
				\node[mycirc,shape=circle,draw=black,text width=0.6cm, label=center:{ 0}] (117) at (3-1,1) {};

				
				\path[draw=black]  (007) edge (107);
				\path[draw=black]  (007) edge (017);
				\path[draw=black]  (117) edge (017);
				\path[draw=black]  (117) edge (107);

				\path[draw=black]  (00) edge (10);
				\path[draw=black]  (00) edge (01);
				\path[draw=black]  (11) edge (01);
				\path[draw=black]  (11) edge (10);

			\end{tikzpicture}

		\end{minipage}\hfill
			\caption{Functions $\varphi_1$ and $\varphi_2$} \label{fig:phi}
	\end{center}

\end{figure}

We define a function $\psi_1$ on $\mathbb{Z}_3^3$ by the following rule:
$$
\psi_1(x)=\begin{cases}
1,&\text{if $x\in \{(0,0,0),(1,1,1),(2,2,2)\}$;}\\
-1,&\text{if $x\in \{(1,2,0),(2,0,1),(0,1,2)\}$;}\\
0,&\text{otherwise.}
\end{cases}
$$

We define a function $\psi_2$ on $\mathbb{Z}_3^3$ by the following rule:
$$
\psi_2(x)=\begin{cases}
1,&\text{if $x\in \{(1,0,0),(2,1,1),(0,2,2)\}$;}\\
-1,&\text{if $x\in \{(2,2,0),(0,0,1),(1,1,2)\}$;}\\
0,&\text{otherwise.}
\end{cases}
$$

We define a function $\psi_3$ on $\mathbb{Z}_3^3$ by the following rule:
$$
\psi_3(x)=\begin{cases}
1,&\text{if $x\in \{(1,1,0),(2,2,1),(0,0,2)\}$;}\\
-1,&\text{if $x\in \{(2,0,0),(0,1,1),(1,2,2)\}$;}\\
0,&\text{otherwise.}
\end{cases}
$$

The functions $\varphi_1$ and $\varphi_2$ are shown in Figure \ref{fig:phi}, 
and the functions $\psi_1$, $\psi_2$ and $\psi_3$ are shown in Figure \ref{fig:psi}.

The following result follows directly from the definition of an eigenfunction.

\begin{lemma}\label{L:Small-Eigenfunctions}

\begin{enumerate}
  The following statements hold:
  
  \item The functions $\varphi_1$ and $\varphi_2$ are $\lambda_1(2,2)$-eigenfunctions of $H(2,2)$.
  
  \item The functions $\psi_1$, $\psi_2$ and $\psi_3$ are $\lambda_2(3,3)$-eigenfunctions of $H(3,3)$.

\end{enumerate}

\end{lemma}

We will also need the following fact.

\begin{lemma}\label{L:U_2(3,3)}
There exists a $\lambda_{2}(3,3)$-eigenfunction $h$ of $H(3,3)$ such that $\mathrm{SND}(h)=2$.
\end{lemma}
\begin{proof}
Let $h=\psi_1+\psi_2+\psi_3$.
Let us show that $h$ satisfies the conditions of the lemma.
By Lemma \ref{L:Small-Eigenfunctions}, the functions $\psi_1$, $\psi_2$ and $\psi_3$ are $\lambda_2(3,3)$-eigenfunctions of $H(3,3)$.
Therefore, $h$ is also a $\lambda_{2}(3,3)$-eigenfunction of $H(3,3)$.
In addition, it is easy to check that the subgraphs of $H(3,3)$ induced by $S_{+}(h)$ and $S_{-}(h)$ are connected.
Therefore, we have $\mathrm{SND}(h)=2$.
\end{proof}

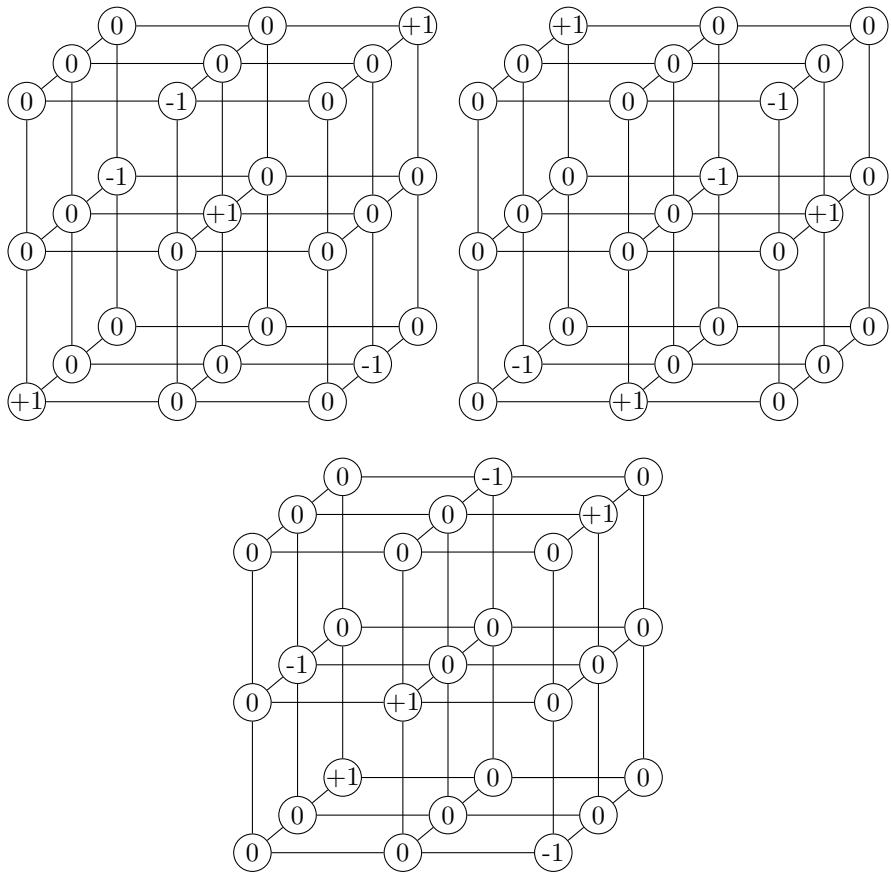
\begin{figure}
	\begin{center}
			
			\begin{tikzpicture}[scale=2,
				mycirc/.style={circle, minimum size=0.2cm}]
				\tikzstyle{every node}=[draw,circle,fill=white,minimum size=0pt,
				inner sep=0pt]
				\usetikzlibrary{math}
				\tikzmath{\x1=0.3; \x2=0.25;}
				
				
				\node[mycirc,shape=circle,draw=black,text width=0.5cm, label=center:{ +1}] (000) at (0,0) {};
				\node[mycirc,shape=circle,draw=black,text width=0.5cm,label=center:{ 0}] (100) at (1,0) {};
				\node[mycirc,shape=circle,draw=black,text width=0.5cm, label=center:{ 0}] (010) at (0,1) {};
				\node[mycirc,shape=circle,draw=black,text width=0.5cm, label=center:{ 0}] (110) at (1,1) {};

				\node[mycirc,shape=circle,draw=black,text width=0.5cm,label=center:{ 0} ] (001) at (0+\x1,0+\x2) {};
				
				\node[mycirc,shape=circle,draw=black,text width=0.5cm, label=center:{ 0}] (101) at (1+\x1,0+\x2) {};
				\node[mycirc,shape=circle,draw=black,text width=0.5cm, label=center:{ 0}] (011) at (0+\x1,1+\x2) {};
				\node[mycirc,shape=circle,draw=black,text width=0.5cm,label=center:{ +1} ] (111) at (1+\x1,1+\x2) {};

			\node[mycirc,shape=circle,draw=black,text width=0.5cm,label=center:{ 0}] (200) at (1+1,0) {};
			\node[mycirc,shape=circle,draw=black,text width=0.5cm,label=center:{ 0} ] (210) at (1+1,1) {};
			\node[mycirc,shape=circle,draw=black,text width=0.5cm, label=center:{ -1}] (201) at (1+\x1+1,0+\x2) {};
			\node[mycirc,shape=circle,draw=black,text width=0.5cm, label=center:{ 0}] (211) at (1+\x1+1,1+\x2) {};

			\node[mycirc,shape=circle,draw=black,text width=0.5cm,label=center:{ 0} ] (020) at (0,1+1) {};
			\node[mycirc,shape=circle,draw=black,text width=0.5cm, label=center:{ -1}] (120) at (1,1+1) {};
			\node[mycirc,shape=circle,draw=black,text width=0.5cm, label=center:{ 0}] (021) at (0+\x1,1+\x2+1) {};
			\node[mycirc,shape=circle,draw=black,text width=0.5cm, label=center:{ 0}] (121) at (1+\x1,1+\x2+1) {};

			\node[mycirc,shape=circle,draw=black,text width=0.5cm, label=center:{ 0}] (002) at (0+\x1+\x1,0+\x2+\x2) {};
			
			\node[mycirc,shape=circle,draw=black,text width=0.5cm,label=center:{ 0} ] (102) at (1+\x1+\x1,0+\x2+\x2) {};
			\node[mycirc,shape=circle,draw=black,text width=0.5cm, label=center:{ -1}] (012) at (0+\x1+\x1,1+\x2+\x2) {};
			\node[mycirc,shape=circle,draw=black,text width=0.5cm, label=center:{ 0}] (112) at (1+\x1+\x1,1+\x2+\x2) {};

			\node[mycirc,shape=circle,draw=black,text width=0.5cm, label=center:{ 0}] (220) at (2,2) {};
			\node[mycirc,shape=circle,draw=black,text width=0.5cm, label=center:{ 0}] (221) at (2+\x1,2+\x2) {};
			\node[mycirc,shape=circle,draw=black,text width=0.5cm,label=center:{ 0} ] (202) at (2+\x1+\x1,\x2+\x2) {};
			\node[mycirc,shape=circle,draw=black,text width=0.5cm, label=center:{ 0}] (212) at (2+\x1+\x1,\x2+\x2+1) {};
			\node[mycirc,shape=circle,draw=black,text width=0.5cm, label=center:{ 0}] (022) at (0+\x1+\x1,2+\x2+\x2) {};
			\node[mycirc,shape=circle,draw=black,text width=0.5cm, label=center:{ 0}] (122) at (1+\x1+\x1,2+\x2+\x2) {};
			
			\node[mycirc,shape=circle,draw=black,text width=0.5cm,label=center:{ +1} ] (222) at (2+\x1+\x1,2+\x2+\x2) {};


				
				\path[draw=black]  (000) edge (100);
				\path[draw=black]  (000) edge (010);
				\path[draw=black]  (000) edge (001);

				\path[draw=black]  (100) edge (200);
				\path[draw=black]  (100) edge (110);	
				\path[draw=black]  (100) edge (101);		
				
				\path[draw=black]  (010) edge (110);	
				\path[draw=black]  (010) edge (011);	
				\path[draw=black]  (010) edge (020);
				
				\path[draw=black]  (001) edge (101);	
				\path[draw=black]  (001) edge (011);	
				\path[draw=black]  (001) edge (002);

				\path[draw=black]  (110) edge (210);
				\path[draw=black]  (110) edge (120);	
				\path[draw=black]  (110) edge (111);	
				
				\path[draw=black]  (101) edge (201);
				\path[draw=black]  (101) edge (111);	
				\path[draw=black]  (101) edge (102);	
				
				\path[draw=black]  (011) edge (111);	
				\path[draw=black]  (011) edge (012);	
				\path[draw=black]  (011) edge (021);
				
				\path[draw=black]  (111) edge (211);
				\path[draw=black]  (111) edge (121);	
				\path[draw=black]  (111) edge (112);

				\path[draw=black]  (200) edge (210);
				\path[draw=black]  (200) edge (201);
				\path[draw=black]  (020) edge (120);
				\path[draw=black]  (020) edge (021);
				\path[draw=black]  (002) edge (102);
				\path[draw=black]  (002) edge (012);
				
				\path[draw=black]  (210) edge (220);
				\path[draw=black]  (210) edge (211);	
				\path[draw=black]  (201) edge (202);
				\path[draw=black]  (201) edge (211);				
				\path[draw=black]  (120) edge (220);
				\path[draw=black]  (120) edge (121);
				\path[draw=black]  (021) edge (022);
				\path[draw=black]  (021) edge (121);
				\path[draw=black]  (102) edge (202);
				\path[draw=black]  (102) edge (112);
				\path[draw=black]  (012) edge (022);
				\path[draw=black]  (012) edge (112);

				\path[draw=black]  (211) edge (221);
				\path[draw=black]  (211) edge (212);

				\path[draw=black]  (121) edge (221);
				\path[draw=black]  (121) edge (122);

				\path[draw=black]  (102) edge (202);
				\path[draw=black]  (112) edge (122);
				\path[draw=black]  (112) edge (212);

				\path[draw=black]  (202) edge (212);
				\path[draw=black]  (220) edge (221);
				\path[draw=black]  (022) edge (122);
				
				\path[draw=black]  (222) edge (122);
				\path[draw=black]  (222) edge (212);
				\path[draw=black]  (222) edge (221);

			\node[mycirc,shape=circle,draw=black,text width=0.5cm, label=center:{ 0}] (0007) at (3+0,0) {};
			\node[mycirc,shape=circle,draw=black,text width=0.5cm,label=center:{ +1}] (1007) at (3+1,0) {};
			\node[mycirc,shape=circle,draw=black,text width=0.5cm, label=center:{ 0}] (0107) at (3+0,1) {};
			\node[mycirc,shape=circle,draw=black,text width=0.5cm, label=center:{ 0}] (1107) at (3+1,1) {};

			\node[mycirc,shape=circle,draw=black,text width=0.5cm,label=center:{ -1} ] (0017) at (3+0+\x1,0+\x2) {};
			
			\node[mycirc,shape=circle,draw=black,text width=0.5cm, label=center:{ 0}] (1017) at (3+1+\x1,0+\x2) {};
			\node[mycirc,shape=circle,draw=black,text width=0.5cm, label=center:{ 0}] (0117) at (3+0+\x1,1+\x2) {};
			\node[mycirc,shape=circle,draw=black,text width=0.5cm,label=center:{0} ] (1117) at (3+1+\x1,1+\x2) {};

			W
			\node[mycirc,shape=circle,draw=black,text width=0.5cm,label=center:{ 0}] (2007) at (3+1+1,0) {};
			\node[mycirc,shape=circle,draw=black,text width=0.5cm,label=center:{ 0} ] (2107) at (3+1+1,1) {};
			\node[mycirc,shape=circle,draw=black,text width=0.5cm, label=center:{ 0}] (2017) at (3+1+\x1+1,0+\x2) {};
			\node[mycirc,shape=circle,draw=black,text width=0.5cm, label=center:{ +1}] (2117) at (3+1+\x1+1,1+\x2) {};

			\node[mycirc,shape=circle,draw=black,text width=0.5cm,label=center:{ 0} ] (0207) at (3+0,1+1) {};
			\node[mycirc,shape=circle,draw=black,text width=0.5cm, label=center:{ 0}] (1207) at (3+1,1+1) {};
			\node[mycirc,shape=circle,draw=black,text width=0.5cm, label=center:{ 0}] (0217) at (3+0+\x1,1+\x2+1) {};
			\node[mycirc,shape=circle,draw=black,text width=0.5cm, label=center:{ 0}] (1217) at (3+1+\x1,1+\x2+1) {};

			\node[mycirc,shape=circle,draw=black,text width=0.5cm, label=center:{ 0}] (0027) at (3+0+\x1+\x1,0+\x2+\x2) {};
			
			\node[mycirc,shape=circle,draw=black,text width=0.5cm,label=center:{ 0} ] (1027) at (3+1+\x1+\x1,0+\x2+\x2) {};
			\node[mycirc,shape=circle,draw=black,text width=0.5cm, label=center:{ 0}] (0127) at (3+0+\x1+\x1,1+\x2+\x2) {};
			\node[mycirc,shape=circle,draw=black,text width=0.5cm, label=center:{ -1}] (1127) at (3+1+\x1+\x1,1+\x2+\x2) {};

			\node[mycirc,shape=circle,draw=black,text width=0.5cm, label=center:{ -1}] (2207) at (3+2,2) {};
			\node[mycirc,shape=circle,draw=black,text width=0.5cm, label=center:{ 0}] (2217) at (3+2+\x1,2+\x2) {};
			\node[mycirc,shape=circle,draw=black,text width=0.5cm,label=center:{ 0} ] (2027) at (3+2+\x1+\x1,\x2+\x2) {};
			\node[mycirc,shape=circle,draw=black,text width=0.5cm, label=center:{ 0}] (2127) at (3+2+\x1+\x1,\x2+\x2+1) {};
			\node[mycirc,shape=circle,draw=black,text width=0.5cm, label=center:{ +1}] (0227) at (3+0+\x1+\x1,2+\x2+\x2) {};
			\node[mycirc,shape=circle,draw=black,text width=0.5cm, label=center:{ 0}] (1227) at (3+1+\x1+\x1,2+\x2+\x2) {};
			
			\node[mycirc,shape=circle,draw=black,text width=0.5cm,label=center:{ 0} ] (2227) at (3+2+\x1+\x1,2+\x2+\x2) {};


			
			\path[draw=black]  (0007) edge (1007);
			\path[draw=black]  (0007) edge (0107);
			\path[draw=black]  (0007) edge (0017);

			\path[draw=black]  (1007) edge (2007);
			\path[draw=black]  (1007) edge (1107);	
			\path[draw=black]  (1007) edge (1017);		
			
			\path[draw=black]  (0107) edge (1107);	
			\path[draw=black]  (0107) edge (0117);	
			\path[draw=black]  (0107) edge (0207);
			
			\path[draw=black]  (0017) edge (1017);	
			\path[draw=black]  (0017) edge (0117);	
			\path[draw=black]  (0017) edge (0027);

			\path[draw=black]  (1107) edge (2107);
			\path[draw=black]  (1107) edge (1207);	
			\path[draw=black]  (1107) edge (1117);	
			
			\path[draw=black]  (1017) edge (2017);
			\path[draw=black]  (1017) edge (1117);	
			\path[draw=black]  (1017) edge (1027);	
			
			\path[draw=black]  (0117) edge (1117);	
			\path[draw=black]  (0117) edge (0127);	
			\path[draw=black]  (0117) edge (0217);
			
			\path[draw=black]  (1117) edge (2117);
			\path[draw=black]  (1117) edge (1217);	
			\path[draw=black]  (1117) edge (1127);

			\path[draw=black]  (2007) edge (2107);
			\path[draw=black]  (2007) edge (2017);
			\path[draw=black]  (0207) edge (1207);
			\path[draw=black]  (0207) edge (0217);
			\path[draw=black]  (0027) edge (1027);
			\path[draw=black]  (0027) edge (0127);
			
			\path[draw=black]  (2107) edge (2207);
			\path[draw=black]  (2107) edge (2117);	
			\path[draw=black]  (2017) edge (2027);
			\path[draw=black]  (2017) edge (2117);				
			\path[draw=black]  (1207) edge (2207);
			\path[draw=black]  (1207) edge (1217);
			\path[draw=black]  (0217) edge (0227);
			\path[draw=black]  (0217) edge (1217);
			\path[draw=black]  (1027) edge (2027);
			\path[draw=black]  (1027) edge (1127);
			\path[draw=black]  (0127) edge (0227);
			\path[draw=black]  (0127) edge (1127);

			\path[draw=black]  (2117) edge (2217);
			\path[draw=black]  (2117) edge (2127);
			
			\path[draw=black]  (1217) edge (2217);
			\path[draw=black]  (1217) edge (1227);
			
			\path[draw=black]  (1027) edge (2027);
			\path[draw=black]  (1127) edge (1227);
			\path[draw=black]  (1127) edge (2127);

			\path[draw=black]  (2027) edge (2127);
			\path[draw=black]  (2207) edge (2217);
			\path[draw=black]  (0227) edge (1227);
			
			\path[draw=black]  (2227) edge (1227);
			\path[draw=black]  (2227) edge (2127);
			\path[draw=black]  (2227) edge (2217);

			
			\node[mycirc,shape=circle,draw=black,text width=0.5cm, label=center:{ 0}] (0008) at (1.5+0,0-3) {};
			\node[mycirc,shape=circle,draw=black,text width=0.5cm,label=center:{ 0}] (1008) at (1.5+1,0-3) {};
			\node[mycirc,shape=circle,draw=black,text width=0.5cm, label=center:{ 0}] (0108) at (1.5+0,1-3) {};
			\node[mycirc,shape=circle,draw=black,text width=0.5cm, label=center:{ +1}] (1108) at (1.5+1,1-3) {};

			\node[mycirc,shape=circle,draw=black,text width=0.5cm,label=center:{ 0} ] (0018) at (1.5+0+\x1,0+\x2-3) {};
			
			\node[mycirc,shape=circle,draw=black,text width=0.5cm, label=center:{ 0}] (1018) at (1.5+1+\x1,0+\x2-3) {};
			\node[mycirc,shape=circle,draw=black,text width=0.5cm, label=center:{ -1}] (0118) at (1.5+0+\x1,1+\x2-3) {};
			\node[mycirc,shape=circle,draw=black,text width=0.5cm,label=center:{ 0} ] (1118) at (1.5+1+\x1,1+\x2-3) {};

			\node[mycirc,shape=circle,draw=black,text width=0.5cm,label=center:{ -1}] (2008) at (1.5+1+1,0-3) {};
			\node[mycirc,shape=circle,draw=black,text width=0.5cm,label=center:{ 0} ] (2108) at (1.5+1+1,1-3) {};
			\node[mycirc,shape=circle,draw=black,text width=0.5cm, label=center:{ 0}] (2018) at (1.5+1+\x1+1,0+\x2-3) {};
			\node[mycirc,shape=circle,draw=black,text width=0.5cm, label=center:{ 0}] (2118) at (1.5+1+\x1+1,1+\x2-3) {};

			\node[mycirc,shape=circle,draw=black,text width=0.5cm,label=center:{ 0} ] (0208) at (1.5+0,1+1-3) {};
			\node[mycirc,shape=circle,draw=black,text width=0.5cm, label=center:{ 0}] (1208) at (1.5+1,1+1-3) {};
			\node[mycirc,shape=circle,draw=black,text width=0.5cm, label=center:{ 0}] (0218) at (1.5+0+\x1,1+\x2+1-3) {};
			\node[mycirc,shape=circle,draw=black,text width=0.5cm, label=center:{ 0}] (1218) at (1.5+1+\x1,1+\x2+1-3) {};

			\node[mycirc,shape=circle,draw=black,text width=0.5cm, label=center:{ +1}] (0028) at (1.5+0+\x1+\x1,0+\x2+\x2-3) {};
			
			\node[mycirc,shape=circle,draw=black,text width=0.5cm,label=center:{ 0} ] (1028) at (1.5+1+\x1+\x1,0+\x2+\x2-3) {};
			\node[mycirc,shape=circle,draw=black,text width=0.5cm, label=center:{ 0}] (0128) at (1.5+0+\x1+\x1,1+\x2+\x2-3) {};
			\node[mycirc,shape=circle,draw=black,text width=0.5cm, label=center:{ 0}] (1128) at (1.5+1+\x1+\x1,1+\x2+\x2-3) {};

			\node[mycirc,shape=circle,draw=black,text width=0.5cm, label=center:{ 0}] (2208) at (1.5+2,2-3) {};
			\node[mycirc,shape=circle,draw=black,text width=0.5cm, label=center:{ +1}] (2218) at (1.5+2+\x1,2+\x2-3) {};
			\node[mycirc,shape=circle,draw=black,text width=0.5cm,label=center:{ 0} ] (2028) at (1.5+2+\x1+\x1,\x2+\x2-3) {};
			\node[mycirc,shape=circle,draw=black,text width=0.5cm, label=center:{ 0}] (2128) at (1.5+2+\x1+\x1,\x2+\x2+1-3) {};
			\node[mycirc,shape=circle,draw=black,text width=0.5cm, label=center:{ 0}] (0228) at (1.5+0+\x1+\x1,2+\x2+\x2-3) {};
			\node[mycirc,shape=circle,draw=black,text width=0.5cm, label=center:{ -1}] (1228) at (1.5+1+\x1+\x1,2+\x2+\x2-3) {};
			
			\node[mycirc,shape=circle,draw=black,text width=0.5cm,label=center:{ 0} ] (2228) at (1.5+2+\x1+\x1,2+\x2+\x2-3) {};


			
			\path[draw=black]  (0008) edge (1008);
			\path[draw=black]  (0008) edge (0108);
			\path[draw=black]  (0008) edge (0018);

			\path[draw=black]  (1008) edge (2008);
			\path[draw=black]  (1008) edge (1108);	
			\path[draw=black]  (1008) edge (1018);		
			
			\path[draw=black]  (0108) edge (1108);	
			\path[draw=black]  (0108) edge (0118);	
			\path[draw=black]  (0108) edge (0208);
			
			\path[draw=black]  (0018) edge (1018);	
			\path[draw=black]  (0018) edge (0118);	
			\path[draw=black]  (0018) edge (0028);

			\path[draw=black]  (1108) edge (2108);
			\path[draw=black]  (1108) edge (1208);	
			\path[draw=black]  (1108) edge (1118);	
			
			\path[draw=black]  (1018) edge (2018);
			\path[draw=black]  (1018) edge (1118);	
			\path[draw=black]  (1018) edge (1028);	
			
			\path[draw=black]  (0118) edge (1118);	
			\path[draw=black]  (0118) edge (0128);	
			\path[draw=black]  (0118) edge (0218);
			
			\path[draw=black]  (1118) edge (2118);
			\path[draw=black]  (1118) edge (1218);	
			\path[draw=black]  (1118) edge (1128);

			\path[draw=black]  (2008) edge (2108);
			\path[draw=black]  (2008) edge (2018);
			\path[draw=black]  (0208) edge (1208);
			\path[draw=black]  (0208) edge (0218);
			\path[draw=black]  (0028) edge (1028);
			\path[draw=black]  (0028) edge (0128);
			
			\path[draw=black]  (2108) edge (2208);
			\path[draw=black]  (2108) edge (2118);	
			\path[draw=black]  (2018) edge (2028);
			\path[draw=black]  (2018) edge (2118);				
			\path[draw=black]  (1208) edge (2208);
			\path[draw=black]  (1208) edge (1218);
			\path[draw=black]  (0218) edge (0228);
			\path[draw=black]  (0218) edge (1218);
			\path[draw=black]  (1028) edge (2028);
			\path[draw=black]  (1028) edge (1128);
			\path[draw=black]  (0128) edge (0228);
			\path[draw=black]  (0128) edge (1128);

			\path[draw=black]  (2118) edge (2218);
			\path[draw=black]  (2118) edge (2128);
			
			\path[draw=black]  (1218) edge (2218);
			\path[draw=black]  (1218) edge (1228);
			
			\path[draw=black]  (1028) edge (2028);
			\path[draw=black]  (1128) edge (1228);
			\path[draw=black]  (1128) edge (2128);

			\path[draw=black]  (2028) edge (2128);
			\path[draw=black]  (2208) edge (2218);
			\path[draw=black]  (0228) edge (1228);
			
			\path[draw=black]  (2228) edge (1228);
			\path[draw=black]  (2228) edge (2128);
			\path[draw=black]  (2228) edge (2218);			
				
			\end{tikzpicture}

		
	\end{center}
	\caption{Functions $\psi_1$, $\psi_2$ and $\psi_3$} \label{fig:psi}
\end{figure}

\section{Case $q=2$}\label{Sec:q=2}

In this section, we consider Problem \ref{Problem:2}. The main result of this section is the following.

\begin{theorem}\label{Th:1}
The following statements hold:

\begin{enumerate}
  
  \item For every $k\geq 1$ and every $n\geq 3k+2$, there exists a $\lambda_{2k+1}(n,2)$-eigenfunction $f$ of $H(n,2)$ such that $\mathrm{SND}(f)=2$.
  
  \item For every $k\geq 1$ and every $n\geq 3k+4$, there exists a $\lambda_{2k+2}(n,2)$-eigenfunction $f$ of $H(n,2)$ such that $\mathrm{SND}(f)=2$.
\end{enumerate}

\end{theorem}

\begin{proof}
1. Let $(C_1,C_2)$ be an equitable $2$-partition of $H(3k,2)$ of type A. 
Let $C_3=C_1+e_t$ and $C_4=C_2+e_t$, where $t\in [3k]$ (we can take any $t$ from $[3k]$).
Let $g$ and $h$ be the functions associated with equitable $2$-partitions $(C_1,C_2)$ and $(C_3,C_4)$ respectively.
We define a function $f$ on $\mathbb{Z}_2^{3k+2}$ by $f=g\otimes \varphi_1+h\otimes \varphi_2$.
Using Corollary \ref{Corollary:Product}, Lemma \ref{L:Associated-Corollary} and Lemma~\ref{L:Small-Eigenfunctions}, we obtain that $g\otimes \varphi_1$ and $h\otimes \varphi_2$ are $\lambda_{2k+1}(3k+2,2)$-eigenfunctions of $H(3k+2,2)$.
Therefore, $f$ is also a $\lambda_{2k+1}(3k+2,2)$-eigenfunction of $H(3k+2,2)$.

Let us prove that the subgraph of $H(3k+2,2)$ induced by $S_{+}(f)$ is connected.
We will need the following five observations:

\begin{itemize}
  \item By Lemma \ref{L:Translation}, we have $C_1\subset C_4$.
Therefore, any vertex from $S_+(f_{0,0})$ has a neighbor in $S_{+}(f_{1,0})$.

  \item By Lemma \ref{L:Connectivity}, the subgraph of $H(3k,2)$ induced by $C_2$ is connected.
Since $C_4=C_2+e_t$, the subgraph of $H(3k,2)$ induced by $C_4$ is also connected.
Hence, the subgraph of $H(3k+2,2)$ induced by $S_{+}(f_{1,0})$ is connected.

  \item Lemma \ref{L:Translation} implies that $C_2\cap C_4\neq \emptyset$. 
Consequently, there is an edge between $S_{+}(f_{1,0})$ and $S_{+}(f_{1,1})$ in $H(3k+2,2)$.

  \item  By Lemma \ref{L:Connectivity}, the subgraph of $H(3k,2)$ induced by $C_2$ is connected.
Therefore, the subgraph of $H(3k+2,2)$ induced by $S_{+}(f_{1,1})$ is connected.

  \item We have $C_3\subset C_2$ due to Lemma \ref{L:Translation}.
Consequently, any vertex from $S_+(f_{0,1})$ has a neighbor in $S_{+}(f_{1,1})$.
\end{itemize}

Combining these observations, we obtain that the subgraph of $H(3k+2,2)$ induced by $S_{+}(f)$ is connected.
Similarly, we can prove that the subgraph of $H(3k+2,2)$ induced by $S_{-}(f)$ is connected. Namely, it is enough to change $S_{+}(f_{a,b})$ to $S_{-}(f_{1-a,1-b})$
in all places in the five observations above. Therefore, we have $\mathrm{SND}(f)=2$.
Thus, we have proved the statement for $n=3k+2$. The proof for $n>3k+2$ follows from Corollary \ref{Cor:N-E}.

2. Let $f$ be the function defined above. 
We define a function $f'$ on $\mathbb{Z}_2^{3k+2}$ by $f'(x)=f(x+e_{3k+2})$.
Since $f$ is a $\lambda_{2k+1}(3k+2,2)$-eigenfunction of $H(3k+2,2)$, $f'$ is also a $\lambda_{2k+1}(3k+2,2)$-eigenfunction of $H(3k+2,2)$.
We define a function $\widetilde{f}$ on $\mathbb{Z}_2^{3k+4}$ by $\widetilde{f}=f\otimes \varphi_1+f'\otimes \varphi_2$.
Using Corollary~\ref{Corollary:Product} and Lemma~\ref{L:Small-Eigenfunctions}, 
we obtain that $\widetilde{f}$ is a $\lambda_{2k+2}(3k+4,2)$-eigenfunction of $H(3k+4,2)$.

Let us prove that the subgraph of $H(3k+4,2)$ induced by $S_{+}(\widetilde{f})$ is connected.
As we proved above, we have $\mathrm{SND}(f)=2$. Therefore,

\begin{equation}\label{Eq:2}
\mathrm{SND}(\widetilde{f}_{0,0})=\mathrm{SND}(\widetilde{f}_{1,1})=2
\end{equation}

Since $f'(x)=f(x+e_{3k+2})$, we have $\mathrm{SND}(f')=2$. Hence,

\begin{equation}\label{Eq:3}
\mathrm{SND}(\widetilde{f}_{1,0})=\mathrm{SND}(\widetilde{f}_{0,1})=2
\end{equation}
We will need the following three observations:

\begin{itemize}
  
  \item  By Lemma \ref{L:Translation}, we have $C_1\subset C_4$.
Therefore, any vertex from $S_+(\widetilde{f}_{0,0,0,0})$ has a neighbor in $S_{+}(\widetilde{f}_{0,0,1,0})$.
  Hence, there are edges between $S_{+}(\widetilde{f}_{0,0})$ and $S_{+}(\widetilde{f}_{1,0})$ in $H(3k+4,2)$.

  \item Lemma \ref{L:Translation} implies that $C_2\cap C_4\neq \emptyset$. 
Consequently, there is an edge between $S_{+}(\widetilde{f}_{0,0,1,0})$ and $S_{+}(\widetilde{f}_{0,0,1,1})$ in $H(3k+4,2)$. 
 Hence, there is an edge between $S_{+}(\widetilde{f}_{1,0})$ and $S_{+}(\widetilde{f}_{1,1})$ in $H(3k+4,2)$.

  \item We have $C_3\subset C_2$ due to Lemma \ref{L:Translation}.
Therefore, any vertex from $S_+(\widetilde{f}_{0,0,0,1})$ has a neighbor in $S_{+}(\widetilde{f}_{0,0,1,1})$. 
 Hence, there are edges between $S_{+}(\widetilde{f}_{0,1})$ and $S_{+}(\widetilde{f}_{1,1})$ in $H(3k+4,2)$.  
  
\end{itemize}

Combining these observations with (\ref{Eq:2}) and (\ref{Eq:3}), we obtain that the subgraph of $H(3k+4,2)$ induced by $S_{+}(\widetilde{f})$ is connected.
Similarly, we can prove that the subgraph of $H(3k+4,2)$ induced by $S_{-}(\widetilde{f})$ is connected. Namely, it is enough to change $S_{+}(\widetilde{f}_{a,b})$ and $S_{+}(\widetilde{f}_{0,0,a,b})$ 
to $S_{-}(\widetilde{f}_{1-a,1-b})$ and $S_{-}(\widetilde{f}_{0,0,1-a,1-b})$ 
in all places in the three observations above. Therefore, we have $\mathrm{SND}(\widetilde{f})=2$.
Thus, we have proved the statement for $n=3k+4$. The proof for $n>3k+4$ follows from Corollary \ref{Cor:N-E}.
\end{proof}

\section{Case $q=3$}\label{Sec:q=3}

In this section, we consider Problem \ref{Problem:3} for $q=3$. The main result of this section is the following.

\begin{theorem}\label{Th:3}
For every $n\geq 2$ and every $i\in [n-1]$, 
there exists a $\lambda_{i}(n,3)$-eigenfunction $f$ of $H(n,3)$ such that $\mathrm{SND}(f)=2$.
\end{theorem}
\begin{proof}
Let $(C_1,C_2)$ be an equitable $2$-partition of $H(n,3)$ of type B, where $n\geq 1$. 
Let $C_3=C_1+e_t$ and $C_4=C_2+e_t$, where $t\in [n]$ (we can take any $t$ from $[n]$).
Let $g$ and $h$ be the functions associated with equitable $2$-partitions $(C_1,C_2)$ and $(C_3,C_4)$.
We define a function $f$ on $\mathbb{Z}_3^{n+3}$ by $f=g\otimes \psi_1+h\otimes \psi_2+h\otimes \psi_3$.
Using Corollary~\ref{Corollary:Product}, Lemma \ref{L:Associated-Corollary} and Lemma \ref{L:Small-Eigenfunctions}, we obtain that $g\otimes \psi_1$, $h\otimes \psi_2$ and $h\otimes \psi_3$ are 
$\lambda_{n+2}(n+3,3)$-eigenfunctions of $H(n+3,3)$.
Therefore, $f$ is also a $\lambda_{n+2}(n+3,3)$-eigenfunction of $H(n+3,3)$.
Let us prove that $\mathrm{SND}(f)=2$.

Firstly, let us prove that the subgraph of $H(n+3,3)$ induced by $S_{+}(f_0)$ is connected.
We will need the following seven observations:

\begin{itemize}
  
  \item  By Lemma \ref{L:Translation-3}, we have $C_1\subset C_4$.
Therefore, any vertex from $S_+(f_{0,0,0})$ has a neighbor in $S_{+}(f_{2,0,0})$.
  
  \item  By Lemma \ref{L:Connectivity-3}, the subgraph of $H(n,3)$ induced by $C_2$ is connected.
Since $C_4=C_2+e_t$, the subgraph of $H(n,3)$ induced by $C_4$ is also connected.
Hence, the subgraph of $H(n+3,3)$ induced by $S_{+}(f_{2,0,0})$ is connected.
  
  \item  Note that any vertex from $S_{+}(f_{2,0,0})$ has a neighbor in $S_{+}(f_{2,2,0})$.
  
  \item Lemma \ref{L:Translation-3} implies that $C_2\cap C_4\neq \emptyset$. 
Consequently, there is an edge between $S_{+}(f_{2,2,0})$ and $S_{+}(f_{1,2,0})$ in $H(n+3,3)$. 
  
  \item By Lemma \ref{L:Connectivity-3}, the subgraph of $H(n,3)$ induced by $C_2$ is connected.
Therefore, the subgraph of $H(n+3,3)$ induced by $S_{+}(f_{1,2,0})$ is connected.
  
  \item We have $C_3\subset C_2$ due to Lemma \ref{L:Translation-3}.
Consequently, any vertex from $S_+(f_{1,1,0})$ has a neighbor in $S_{+}(f_{1,2,0})$.

\item  Note that any vertex from $S_{+}(f_{1,0,0})$ has a neighbor in $S_{+}(f_{1,1,0})$.

\end{itemize}

Combining these observations, we obtain that the subgraph of $H(n+3,3)$ induced by $S_{+}(f_0)$ is connected.

Secondly, let us prove that the subgraph of $H(n+3,3)$ induced by $S_{+}(f_1)$ is connected.
We will need the following seven observations:

\begin{itemize}
  
  \item  By Lemma \ref{L:Translation-3}, we have $C_1\subset C_4$.
Therefore, any vertex from $S_+(f_{1,1,1})$ has a neighbor in $S_{+}(f_{0,1,1})$.
  
  \item  By Lemma \ref{L:Connectivity-3}, the subgraph of $H(n,3)$ induced by $C_2$ is connected.
Since $C_4=C_2+e_t$, the subgraph of $H(n,3)$ induced by $C_4$ is also connected.
Hence, the subgraph of $H(n+3,3)$ induced by $S_{+}(f_{0,1,1})$ is connected.
  
  \item  Note that any vertex from $S_{+}(f_{0,1,1})$ has a neighbor in $S_{+}(f_{0,0,1})$.
  
  \item Lemma \ref{L:Translation-3} implies that $C_2\cap C_4\neq \emptyset$. 
Consequently, there is an edge between $S_{+}(f_{0,0,1})$ and $S_{+}(f_{2,0,1})$ in $H(n+3,3)$. 
  
  \item By Lemma \ref{L:Connectivity-3}, the subgraph of $H(n,3)$ induced by $C_2$ is connected.
Therefore, the subgraph of $H(n+3,3)$ induced by $S_{+}(f_{2,0,1})$ is connected.
  
  \item We have $C_3\subset C_2$ due to Lemma \ref{L:Translation-3}.
Consequently, any vertex from $S_+(f_{2,1,1})$ has a neighbor in $S_{+}(f_{2,0,1})$.

\item  Note that any vertex from $S_{+}(f_{2,2,1})$ has a neighbor in $S_{+}(f_{2,1,1})$.

\end{itemize}

Combining these observations, we obtain that the subgraph of $H(n+3,3)$ induced by $S_{+}(f_1)$ is connected.

Thirdly, let us prove that the subgraph of $H(n+3,3)$ induced by $S_{+}(f_2)$ is connected.
We will need the following six observations:

\begin{itemize}
  
  \item  By Lemma \ref{L:Translation-3}, we have $C_1\subset C_4$.
Therefore, any vertex from $S_+(f_{2,2,2})$ has a neighbor in $S_{+}(f_{1,2,2})$.
  
  \item  By Lemma \ref{L:Connectivity-3}, the subgraph of $H(n,3)$ induced by $C_2$ is connected.
Since $C_4=C_2+e_t$, the subgraph of $H(n,3)$ induced by $C_4$ is also connected.
Hence, the subgraph of $H(n+3,3)$ induced by $S_{+}(f_{1,2,2})$ is connected.
  
  \item  Note that any vertex from $S_{+}(f_{1,1,2})$ has a neighbor in $S_{+}(f_{1,2,2})$.
  
  \item Lemma \ref{L:Translation-3} implies that $C_2\cap C_4\neq \emptyset$. 
Consequently, there is an edge between $S_{+}(f_{1,1,2})$ and $S_{+}(f_{0,1,2})$ in $H(n+3,3)$. 
  
  \item By Lemma \ref{L:Connectivity-3}, the subgraph of $H(n,3)$ induced by $C_2$ is connected.
Therefore, the subgraph of $H(n+3,3)$ induced by $S_{+}(f_{0,1,2})$ is connected.
  
  \item We have $C_3\subset C_2$ due to Lemma \ref{L:Translation-3}.
Consequently, any vertex from $S_+(f_{0,2,2})$ has a neighbor in $S_{+}(f_{0,1,2})$.
The same argument implies that any vertex from $S_+(f_{0,0,2})$ has a neighbor in $S_{+}(f_{0,1,2})$.

\end{itemize}

Combining these observations, we obtain that the subgraph of $H(n+3,3)$ induced by $S_{+}(f_2)$ is connected.

As we proved above, the subgraphs of $H(n+3,3)$ induced by $S_{+}(f_0)$, $S_{+}(f_1)$ and  $S_{+}(f_2)$ are connected.
On the other hand, by Lemma \ref{L:Translation-3}, we have $C_2\cap C_4\neq \emptyset$. 
Therefore, there is an edge between $S_{+}(f_{2,0,0})$ and $S_{+}(f_{2,0,1})$ in $H(n+3,3)$.
The same argument shows that there is an edge between $S_{+}(f_{0,1,1})$ and $S_{+}(f_{0,1,2})$ in $H(n+3,3)$.
This implies that the subgraph of $H(n+3,3)$ induced by $S_{+}(f)$ is connected.
Similarly, we can prove that the subgraph of $H(n+3,3)$ induced by $S_{-}(f)$ is connected.
Therefore, we have $\mathrm{SND}(f)=2$.
Thus, we have proved the theorem for $i=n-1$ and $n\geq 4$. 
The proof for $3\leq i<n-1$ and $n\geq 5$ follows from Corollary \ref{Cor:N-E}.

Finally, we consider the case $n\leq 3$. 
The proof for $i=1$ and $n\in \{2,3\}$ follows from Lemma \ref{L:U_1}, 
and the proof for $i=2$ and $n=3$ follows from Lemma~\ref{L:U_2(3,3)}.
\end{proof}

\section{Case $q\geq 4$}\label{Sec:q>=4}

In this section, we consider Problem \ref{Problem:3} for $q\geq 4$. The main result of this section is the following.

\begin{theorem}\label{Th:>3}
Let $n\geq 1$ and $q\geq 4$.
Then for every $i\in [n]$ there exists a $\lambda_{i}(n,q)$-eigenfunction $f$ of $H(n,q)$ such that $\mathrm{SND}(f)=2$.
\end{theorem}
\begin{proof}


Let us consider a $2$-partition $(C_1,C_2)$ of $\mathbb{Z}_q^n$, where $C_1=\Gamma_0^{n,q}\cup \Gamma_1^{n,q}$ and
$C_2=\Gamma_2^{n,q}\cup\Gamma_3^{n,q}\cup\ldots \cup\Gamma_{q-1}^{n,q}$.
One can verify that $(C_1,C_2)$ is an equitable $2$-partition of $H(n,q)$ 
with the quotient matrix 
$\begin{pmatrix}
n & n(q-2)\\
2n & n(q-3)\\
\end{pmatrix}$.
Let $f$ be the associated function of $(C_1,C_2)$.
By Lemma \ref{L:Associated-Function}, $f$ is a $\lambda_{n}(n,q)$-eigenfunction of $H(n,q)$.

Let us prove that $\mathrm{SND}(f)=2$ for all $q\geq 4$.
Note that $S_{+}(f)=\Gamma_0^{n,q}\cup \Gamma_1^{n,q}$.
Hence, by Lemma \ref{L:Hyperplane}, the subgraph of $H(n,q)$ induced by $S_{+}(f)$ is connected.
We also have $S_{-}(f)=\Gamma_2^{n,q}\cup\Gamma_3^{n,q}\cup\ldots \cup\Gamma_{q-1}^{n,q}$.
By Lemma \ref{L:Hyperplane}, the subgraph of $H(n,q)$ induced by $\Gamma_{a}^{n,q}\cup \Gamma_{a+1}^{n,q}$ is connected for all $a\in \mathbb{Z}_q$.
This implies that the subgraph of $H(n,q)$ induced by $S_{-}(f)$ is also connected.
Therefore, we have $\mathrm{SND}(f)=2$.
Thus, we have proved the theorem for $i=n$. The proof for $i<n$ follows from Corollary \ref{Cor:N-E}.
\end{proof}

\section{Concluding remarks}\label{Sec:Conclusion}
In this paper, we consider a problem of finding eigenfunctions of Hamming graphs with the smallest number of strong nodal domains, 
continuing the research from \cite{BHLPS04}.

As the main result we confirm Conjecture \ref{Conj:Strong-H} for all $i$ up to approximately $\frac{2n}{3}$.
This upper bound is closely related to the correlation immunity bound for Boolean functions proved by Fon-Der-Flaass in \cite{FDF07Bound}, 
since our approach is to build required eigenfunctions from known equitable $2$-partitions attaining the bound. 
Thus, it seems that for the remaining cases a fundamentally new approach is needed.

Besides, we consider the problem for $q\geq 3$ for the first time. We prove analogs of Conjecture \ref{Conj:Strong-H} in the following cases: 
\begin{itemize}
	\item $q\geq 4$ and $1\leq i\leq n$,
	\item $q=3$ and $1\leq i\leq n-1$.
\end{itemize}

For $q=3$ and $q=2$ we use similar approaches. 
The case $q\geq 4$ is simpler because in this case there are $\lambda_{n}(n,q)$-eigenfunctions of $H(n,q)$
with two strong nodal domains.

In other words, the last remaining open case for $q\geq 3$ is $q=3$ and $i=n$.
In our opinion, this case has something in common with the case $q=2$, $i=n-1$, 
and the minimum number of strong nodal domains seems to be a linear on $n$ function.
In particular, we conducted numerical experiments for $n=2,3,4$, $i=n$ and the functions with minimum number of strong nodal domains that we found have $3$, $4$ and $5$ strong nodal domains respectively (see Figure \ref{fig:concl}). Based on these computations, we formulate the following conjecture.

\begin{conjecture}
For any eigenfunction $f$ of $H(n,3)$, $n\geq 1$, with eigenvalue $3n$ we have $\mathrm{SND}(f)\geq n+1$.
\end{conjecture}

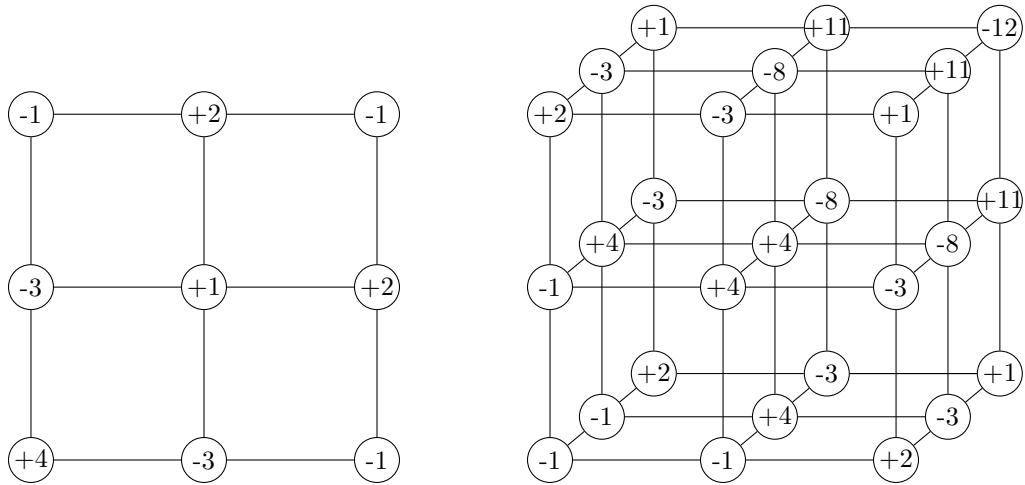
\begin{figure}
	\begin{center}
		\begin{minipage}{1.45\linewidth}
			
			\begin{tikzpicture}[scale=2.3,
				mycirc/.style={circle, minimum size=0.2cm}]
				\tikzstyle{every node}=[draw,circle,fill=white,minimum size=0pt,
				inner sep=0pt]
				\usetikzlibrary{math}
				

				\node[mycirc,shape=circle,draw=black,text width=0.6cm, label=center:{ +4}] (0009) at (0,0) {};
				\node[mycirc,shape=circle,draw=black,text width=0.6cm,label=center:{ -3}] (1009) at (1,0) {};
				\node[mycirc,shape=circle,draw=black,text width=0.6cm, label=center:{ -3}] (0109) at (0,1) {};
				\node[mycirc,shape=circle,draw=black,text width=0.6cm, label=center:{ +1}] (1109) at (1,1) {};
				\node[mycirc,shape=circle,draw=black,text width=0.6cm,label=center:{ -1}] (2009) at (1+1,0) {};
				\node[mycirc,shape=circle,draw=black,text width=0.6cm,label=center:{ +2} ] (2109) at (1+1,1) {};
				
				\node[mycirc,shape=circle,draw=black,text width=0.6cm,label=center:{ -1} ] (0209) at (0,1+1) {};
				\node[mycirc,shape=circle,draw=black,text width=0.6cm, label=center:{ +2}] (1209) at (1,1+1) {};
				\node[mycirc,shape=circle,draw=black,text width=0.6cm, label=center:{ -1}] (2209) at (2,2) {};

				\path[draw=black]  (0009) edge (1009);
				\path[draw=black]  (0009) edge (0109);				
				\path[draw=black]  (1009) edge (2009);
				\path[draw=black]  (1009) edge (1109);					
				\path[draw=black]  (0109) edge (1109);		
				\path[draw=black]  (0109) edge (0209);				
				\path[draw=black]  (1109) edge (2109);
				\path[draw=black]  (1109) edge (1209);		
				\path[draw=black]  (2009) edge (2109);
				\path[draw=black]  (0209) edge (1209);
			
				\path[draw=black]  (2109) edge (2209);		
				\path[draw=black]  (1209) edge (2209);

				\node[mycirc,shape=circle,draw=black,text width=0.6cm, label=center:{ -1}] (000) at (3+0,0) {};
				\node[mycirc,shape=circle,draw=black,text width=0.6cm,label=center:{-1 }] (100) at (3+1,0) {};
				\node[mycirc,shape=circle,draw=black,text width=0.6cm, label=center:{-1 }] (010) at (3+0,1) {};
				\node[mycirc,shape=circle,draw=black,text width=0.6cm, label=center:{+4 }] (110) at (3+1,1) {};

				\node[mycirc,shape=circle,draw=black,text width=0.6cm,label=center:{-1 } ] (001) at (3+0+0.3,0+0.25) {};
				
				\node[mycirc,shape=circle,draw=black,text width=0.6cm, label=center:{+4 }] (101) at (3+1+0.3,0+0.25) {};
				\node[mycirc,shape=circle,draw=black,text width=0.6cm, label=center:{+4 }] (011) at (3+0+0.3,1+0.25) {};
				\node[mycirc,shape=circle,draw=black,text width=0.6cm,label=center:{+4 } ] (111) at (3+1+0.3,1+0.25) {};

				\node[mycirc,shape=circle,draw=black,text width=0.6cm,label=center:{+2 }] (200) at (3+1+1,0) {};
				\node[mycirc,shape=circle,draw=black,text width=0.6cm,label=center:{ -3} ] (210) at (3+1+1,1) {};
				\node[mycirc,shape=circle,draw=black,text width=0.6cm, label=center:{-3 }] (201) at (3+1+0.3+1,0+0.25) {};
				\node[mycirc,shape=circle,draw=black,text width=0.6cm, label=center:{-8 }] (211) at (3+1+0.3+1,1+0.25) {};

				\node[mycirc,shape=circle,draw=black,text width=0.6cm,label=center:{+2} ] (020) at (3+0,1+1) {};
				\node[mycirc,shape=circle,draw=black,text width=0.6cm, label=center:{ -3}] (120) at (3+1,1+1) {};
				\node[mycirc,shape=circle,draw=black,text width=0.6cm, label=center:{-3 }] (021) at (3+0+0.3,1+0.25+1) {};
				\node[mycirc,shape=circle,draw=black,text width=0.6cm, label=center:{ -8}] (121) at (3+1+0.3,1+0.25+1) {};

				\node[mycirc,shape=circle,draw=black,text width=0.6cm, label=center:{+2 }] (002) at (3+0+0.3+0.3,0+0.25+0.25) {};
			
				\node[mycirc,shape=circle,draw=black,text width=0.6cm,label=center:{-3 } ] (102) at (3+1+0.3+0.3,0+0.25+0.25) {};
				\node[mycirc,shape=circle,draw=black,text width=0.6cm, label=center:{ -3}] (012) at (3+0+0.3+0.3,1+0.25+0.25) {};
				\node[mycirc,shape=circle,draw=black,text width=0.6cm, label=center:{-8 }] (112) at (3+1+0.3+0.3,1+0.25+0.25) {};

				\node[mycirc,shape=circle,draw=black,text width=0.6cm, label=center:{ +1}] (220) at (3+2,2) {};
				\node[mycirc,shape=circle,draw=black,text width=0.6cm, label=center:{+11 }] (221) at (3+2+0.3,2+0.25) {};
				\node[mycirc,shape=circle,draw=black,text width=0.6cm,label=center:{ +1} ] (202) at (3+2+0.3+0.3,0.25+0.25) {};
				\node[mycirc,shape=circle,draw=black,text width=0.6cm, label=center:{ +11}] (212) at (3+2+0.3+0.3,0.25+0.25+1) {};
				\node[mycirc,shape=circle,draw=black,text width=0.6cm, label=center:{+1 }] (022) at (3+0+0.3+0.3,2+0.25+0.25) {};
				\node[mycirc,shape=circle,draw=black,text width=0.6cm, label=center:{+11 }] (122) at (3+1+0.3+0.3,2+0.25+0.25) {};
				
				\node[mycirc,shape=circle,draw=black,text width=0.6cm,label=center:{-12 } ] (222) at (3+2+0.3+0.3,2+0.25+0.25) {};


				
				\path[draw=black]  (000) edge (100);
				\path[draw=black]  (000) edge (010);
				\path[draw=black]  (000) edge (001);

				\path[draw=black]  (100) edge (200);
				\path[draw=black]  (100) edge (110);	
				\path[draw=black]  (100) edge (101);		
				
				\path[draw=black]  (010) edge (110);	
				\path[draw=black]  (010) edge (011);	
				\path[draw=black]  (010) edge (020);
				
				\path[draw=black]  (001) edge (101);	
				\path[draw=black]  (001) edge (011);	
				\path[draw=black]  (001) edge (002);

				\path[draw=black]  (110) edge (210);
				\path[draw=black]  (110) edge (120);	
				\path[draw=black]  (110) edge (111);	
				
				\path[draw=black]  (101) edge (201);
				\path[draw=black]  (101) edge (111);	
				\path[draw=black]  (101) edge (102);	
				
				\path[draw=black]  (011) edge (111);	
				\path[draw=black]  (011) edge (012);	
				\path[draw=black]  (011) edge (021);
				
				\path[draw=black]  (111) edge (211);
				\path[draw=black]  (111) edge (121);	
				\path[draw=black]  (111) edge (112);

				\path[draw=black]  (200) edge (210);
				\path[draw=black]  (200) edge (201);
				\path[draw=black]  (020) edge (120);
				\path[draw=black]  (020) edge (021);
				\path[draw=black]  (002) edge (102);
				\path[draw=black]  (002) edge (012);
				
				\path[draw=black]  (210) edge (220);
				\path[draw=black]  (210) edge (211);	
				\path[draw=black]  (201) edge (202);
				\path[draw=black]  (201) edge (211);				
				\path[draw=black]  (120) edge (220);
				\path[draw=black]  (120) edge (121);
				\path[draw=black]  (021) edge (022);
				\path[draw=black]  (021) edge (121);
				\path[draw=black]  (102) edge (202);
				\path[draw=black]  (102) edge (112);
				\path[draw=black]  (012) edge (022);
				\path[draw=black]  (012) edge (112);

				\path[draw=black]  (211) edge (221);
				\path[draw=black]  (211) edge (212);
				
				\path[draw=black]  (121) edge (221);
				\path[draw=black]  (121) edge (122);
				
				\path[draw=black]  (102) edge (202);
				\path[draw=black]  (112) edge (122);
				\path[draw=black]  (112) edge (212);

				\path[draw=black]  (202) edge (212);
				\path[draw=black]  (220) edge (221);
				\path[draw=black]  (022) edge (122);
				
				\path[draw=black]  (222) edge (122);
				\path[draw=black]  (222) edge (212);
				\path[draw=black]  (222) edge (221);

			\end{tikzpicture}

		\end{minipage}\hfill
		
	\end{center}
	\caption{Eigenfunctions with $3$ and $4$ strong nodal domains in $H(2,3)$ and $H(3,3)$ respectively } \label{fig:concl}
\end{figure}

\section{Acknowledgements}
The research of the second author was supported by the NSP P. Beron project CP-MACT.


\end{document}